\documentclass[11pt]{amsart}

\usepackage[utf8x]{inputenc}
\usepackage[T1]{fontenc}
\usepackage{amsmath}
\usepackage{amssymb}
\usepackage{amsthm}
\usepackage[foot]{amsaddr}
\usepackage{enumerate}
\usepackage{fullpage}
\usepackage{graphicx}
\usepackage{lmodern}
\usepackage{xcolor}
\usepackage{diagbox}

\theoremstyle{plain}
\newtheorem{theorem}{Theorem}[section]

\newtheorem{cor}{Corollary}[section]
\newtheorem{lemma}{Lemma}[section]
\numberwithin{equation}{section}

\theoremstyle{definition}
\newtheorem{mydef}{Definition}[section]

\newtheorem{rem}{Remark}

\newcommand{\algns}[1]{\begin{align*} #1 \end{align*}}
\newcommand{\algnd}[1]{\begin{aligned} #1 \end{aligned}}
\newcommand{\eqn}[1]{\begin{equation} #1 \end{equation}}
\newcommand{\eqns}[1]{\begin{equation*} #1 \end{equation*}}

\newcommand{\setof}[1]{\left\{ #1 \right\} }	
\newcommand{\sothat}{\,:\,}
\newcommand{\dual}[1]{#1^{'}}
\newcommand{\adjoint}[1]{#1^*}
\newcommand{\mat}[1]{\mathsf{#1}}
\renewcommand{\vec}[1]{\boldsymbol{#1}}
\newcommand{\dualvec}[1]{\tilde{\vec{#1}}}

\newcommand{\intd}{\mathrm{d}}	
\newcommand{\reals}{\mathbb{R}}

\newcommand{\grad}{\nabla}
\newcommand{\pd}{\partial}

\newcommand{\norm}[1]{\left\Vert#1\right\Vert}
\newcommand{\norw}[2]{\left\Vert#1\right\Vert_{#2}}
\newcommand{\inner}[2]{\left( #1, #2\right)}

\newcommand{\interpolate}[3]{\left[#1,#2\right]_{#3}}

\renewcommand{\div}{\operatorname{div}\,}
\newcommand{\curl}{\boldsymbol{\operatorname{curl}}\,}

\newcommand{\T}{\mathcal{T}}	
\newcommand{\N}{\mathcal{N}}	

\begin{document}
	\title[]{An auxiliary space preconditioner for fractional Laplacian of negative order}\thanks{The research leading to these results has received funding the European Research Council under the European Union's Seventh Framework Programme (FP7/2007-2013) / ERC grant agreement no. 339643.}
	\author{Trygve B\ae rland$^\dagger$}
	\email{trygveba@math.uio.no}
	\address{$^\dagger$Department of Mathematics, University of Oslo, Blindern, Oslo, 0316 Norway}
	
	\begin{abstract}
		Coupled multiphysics problems often give rise to interface conditions naturally formulated in fractional Sobolev spaces. Here, both positive and negative fractionality are common. When designing efficient solvers for discretizations of such problems it would then be useful to have a preconditioner for the fractional Laplacian, $(-\Delta)^s$, with $s \in [-1,1]$. 
		Previously, additive multigrid preconditioners for the case when $s \geq 0$ have been proposed. In this work we complement this construction with auxiliary space preconditioners suitable when $s \leq 0$. 
		These preconditioners are shown to be spectrally equivalent to $(-\Delta)^{-s}$, but requires preconditioners for fractional $H(\div)$ operators with positive fractionality. We design such operators based on an additive multigrid approach.
		We finish with some numerical experiments, verifying the theoretical results.
	\end{abstract}
	\maketitle
	
	\section{Introduction}
	\label{sec:introduction}
	In this paper we are concerned with the design and analysis
	of preconditioners for the fractional Laplacian with negative exponent.
	More specifically, let $\Omega \subset \reals^n$ be a bounded $n$-dimensional domain, and $s \in [0,1]$ a parameter.
	We then consider the problem of finding $u$ satisfying
	\eqn{
		\label{eq:fraclap-problem}
		(-\Delta)^{-s} u = f,
	}
	where $f$ is given. 
	Here, $H_0^1(\Omega)$ denotes the usual Sobolev space of square-integrable functions with
	square-integrable first order derivatives and zero trace on the boundary of $\Omega$, and $H^{-1}(\Omega)$ denotes its dual space.
	Then $(-\Delta)^{-s}$ is defined from the spectral decomposition
	of $(-\Delta): H_0^1(\Omega) \to H^{-1}(\Omega)$.
	Our aim in this work is to design efficient preconditioners for discretizations of $(-\Delta)^{-s}$.
	
	Due to the negative exponent, common preconditioning strategies will fail in this context.
	In particular, for positive $s$, $(-\Delta)^s$ behaves similarly to $-\Delta$
	in that the eigenfunctions corresponding to high eigenvalues are oscillatory, and vice versa. 
	As such, the error from simple iteration schemes, like Richardson's iteration, 
	are relatively smooth and can be well-represented on a coarser function space. 
	This observation suggests that multigrid operators can provide efficient preconditioners for $(-\Delta)^s$,
	and motivated the construction of additive multigrid preconditioners in \cite{baerland2018multigrid}.
	However, in our current context the roles are reversed. 
	The oscillatory eigenfunctions of $(-\Delta)^{-s}$ correspond to the lower end of the spectrum. 
	Then, neither simple smoothing procedures nor coarse grid correction will eliminate the oscillatory part of the error, 
	and therefore we cannot hope for a straightforward multigrid method to work. 
	
	The preconditioners proposed in this work will be based on the auxiliary space preconditioner framework, \cite{xu1996auxiliary}.
	Of particular note is that the transfer operator,
	whose role is to relate the original space and the auxiliary space,
	will be a differential operator.
	Consequently, the preconditioner on the auxiliary space will have to be spectrally equivalent to
	the inverse of a differential operator raised to a positive, fractional power.
	To motivate this, let $H_0^s(\Omega)$ denote the spectral interpolation (see \cite[Ch. 2]{lions1972nonhom}) between $L^2(\Omega)$ and $H_0^1(\Omega)$, and $H^{-s}(\Omega)$ the dual space of $H_0^s(\Omega)$.
	Then $(-\Delta)^{-s}$ is an isomorphism from $H^{-s}(\Omega)$ to $H_0^s(\Omega)$.
	Following the operator preconditioning framework in \cite{mardal2011preconditioning},
	an efficient preconditioner for \eqref{eq:fraclap-problem} should be based on 
	a linear, symmetric isomorphism $B^s: H_0^s(\Omega) \to H^{-s}(\Omega)$,
	the canonical choice being the Riesz mapping $(-\Delta)^s$.
	Consequently, the preconditioner should behave like a differential operator raised to a positive, fractional power.
	Then, roughly speaking, if $B^s$ consists of applications of any standard differential operator, 
	a correction is needed to compensate for this overshoot in fractionality. 
	This correction will then behave like the inverse of a fractional differential operator of positive order. 
	In particular, we will see that $B^s = -\div \Lambda^{-(1-s)} \grad$ is spectrally equivalent to $(-\Delta)^{s}$. 
	Here, $\Lambda = I - \grad \div$ is the operator realizing the $H(\div)$ inner product. 
	Thus, the problem of preconditioning $(-\Delta)^{-s}$ will be 
	transferred to the problem of preconditioning $\Lambda^{1-s}$, 
	which is amenable to an analysis similar to the one made in \cite{baerland2018multigrid}.
	This is an attractive idea because, as we will see, $\Lambda^{1-s}$ behaves similarly to
	$\Lambda$, where preconditioning strategies based on multilevel decompositions
	have proved efficient, \cite{arnold1997preconditioning,arnold2000multigrid,hiptmair1997multigrid,hiptmair2007nodal,vassilevski2012parAMG,vassilevski2014construction}.
	
	Preconditioners, and in particular preconditioners based on multilevel decompositions, for \eqref{eq:fraclap-problem}
	have previously been studied.
	For $s = \frac{1}{2}$,  Bramble et al. designed a V-cycle multigrid operator in \cite{bramble1994analysis}. 
	Their construction was based on posing \eqref{eq:fraclap-problem} in the weaker $H^{-1}$ inner product, 
	where the operator they considered had spectral properties suitable for multigrid analysis. 
	In \cite{funken1997bpx}, similar ideas were used to construct and analyze an additive multigrid operator.
	Hierarchical basis preconditioners, suitable for \eqref{eq:fraclap-problem} when $s \in\left(-\frac{3}{2}, \frac{3}{2} \right)$ 
	were constructed in \cite{oswald1998multilevel}. 
	These preconditioners were based on an $L^2$-orthogonal decomposition into each level
	of the grid hierarchy, 
	and thus restricting its use to wavelet spaces where such decompositions are feasible. 
	This was remedied for finite element spaces of low order in \cite{bramble2000computational} 
	by replacing $L^2$-projections onto each level by more cheaply computed operators. 
	In all the preconditioners mentioned above, one drawback is that only simple scaling smoothers can be used, 
	which might be seen as too restrictive. 
	Lastly, in \cite{stevenson2018optimal} the authors constructed optimal auxiliary space preconditioners 
	for \eqref{eq:fraclap-problem}, but they needed to presuppose that a discrete version of $(-\Delta)^{s}$
	was easily computable in the auxiliary space.
	We will in this work not assume such a discrete operator to be at our disposable.
	That is, the proposed preconditioners will not require the computation of $(-\Delta)^{\pm s}$,
	or the fractional power of any positive definite operator for that matter.
	
	The reason for this design choice is that our main motivational application are
	coupled multihysics- and trace constraint problems, where fractional Sobolev spaces
	are part of a well-posed variational formulation, 
	but the fractional Laplacian is absent from the operator characterizing the problem.
	As an illustrative example, let $\Omega$ be a bounded domain $\reals^n$, with $n=2$ or $3$, 
	and $\Gamma$ denotes a structure in $\Omega$ or on its boundary with codimension $1$. 
	Consider the Poisson equation, $-\Delta u = f$ in $\Omega$, 
	with the constraint conditions $u = g$ on $\Gamma$ for given data $f$ and $g$. 
	Imposing the trace constraint weakly, similarly to how it was done in \cite{babuska1973lagrangian}, 
	yields a saddle point system of the form
	\eqn{
		\label{eq:trace-prob-positive}
		\algnd{
			-\Delta u + \adjoint{T}\lambda &= f, & & x \in \Omega \\
			Tu &= g, & & x \in \Gamma
		}
	}
	where $T: H^1(\Omega) \to H^{\frac{1}{2}}(\Gamma)$ is the trace operator. 
	The solution $(u,\lambda)$ is sought in $H^1(\Omega) \times H^{-\frac{1}{2}}(\Gamma)$. 
	Rewriting \eqref{eq:trace-prob-positive} in matrix form, we have
	\eqns{
		\mathcal{A}\begin{pmatrix}
			u \\
			\lambda
		\end{pmatrix}
		= \begin{pmatrix}
			f \\
			g
		\end{pmatrix},
	}
	where $\mathcal{A} = \begin{pmatrix}
	-\Delta & \adjoint{T} \\
	T & 0
	\end{pmatrix}
	$ is an isomorphism from $H^1(\Omega) \times H^{-\frac{1}{2}}(\Gamma)$ to $\dual{\left(H^1(\Omega)\right)} \times H^{\frac{1}{2}}(\Gamma)$.
	By the framework in \cite{mardal2011preconditioning}, 
	a preconditioner for a discretization of \eqref{eq:trace-prob-positive} should be based on a symmetric isomorphism 
	$\mathcal{B}: \dual{\left(H^1(\Omega)\right)} \times H^{\frac{1}{2}}(\Gamma) \to H^1(\Omega) \times H^{-\frac{1}{2}}(\Gamma)$, 
	with the canonical choice being
	\eqn{
		\label{eq:traceprob-canon-precon}
		\mathcal{B} = \begin{pmatrix}
			(I-\Delta)^{-1} & 0 \\
			0 & (-\Delta_\Gamma)^{\frac{1}{2}}
		\end{pmatrix}.
	}
	Cheaply computable operators, spectrally equivalent to $(I-\Delta)^{-1}$ are well known. 
	The second block, $(-\Delta_\Gamma)^{\frac{1}{2}}$ is as such the challenging part when designing preconditioners based on \eqref{eq:traceprob-canon-precon}.
	See also that the fractional Laplacian only appears in $\mathcal{B}$, and not in $\mathcal{A}$.
	
	We remark that even if the above example is relatively simple,
	similar techniques can be used in problems where different PDEs are posed on
	separate domains and linked through some continuity conditions on a
	common interface $\Gamma$.
	One or more of these continuity conditions can then be enforced weakly
	by use of Lagrange multipliers,
	which often will posed in a fractional Sobolev space.
	When preconditioning the resultant system,
	the problem of establishing a computationally feasible operator,
	spectrally equivalent to $(-\Delta_\Gamma)^{\pm\frac{1}{2}}$ persists.
	For instance, in \cite{kuchta2016trace2d} the authors study a multiphysics problem posed 
	on domains of different topological dimension, 
	and continuity is imposed weakly using a Lagrange multiplier. 
	Other applications can be found in \cite{BERTOLUZZA201758}, 
	where the no-slip condition on the surface of a falling body in a fluid is imposed weakly, 
	or in \cite{tveito2017cell}, where the potential jump on a membrane of a cardiac cell is treated similarly. 
	If the embedded structure $\Gamma$ in \eqref{eq:trace-prob-positive} instead has codimension $2$, 
	then numerical experiments in \cite{kuchta2016trace3d} suggests that block diagonal preconditioners 
	where one block is based on $(-\Delta_\Gamma)^{-s}$, with $s\in (-0.2,-0.1)$, 
	provide efficient preconditioners.
	
	The current paper can in a couple of ways be viewed as continuation of \cite{baerland2018multigrid}. 
	Firstly, we define efficient preconditioners for the fractional Laplacian when the exponent $s\in [-1,0]$, 
	complementing the preconditioners introduced in the previous work. 
	Secondly, in this work we generalize the results from \cite{baerland2018multigrid} 
	to positive fractional powers of $\Lambda$. 
	The analysis will aim to substantiate the intuition that if additive multilevel methods are efficient for $s=0$ and $s=1$, 
	then ``by interpolation'' it should be efficient for every $s\in (0,1)$.
	We remark, however, that the analysis on these multilevel methods for
	fractional $H(\div)$ operators assumes certain two-level error estimates
	on $\Lambda^{1-s}$ that will go unproven in this work.
	This is an unsatisfactory state of affairs, 
	but we do give an approach for how these error estimates can be proven,
	as well as motivate their veracity. 
	The techniques we propose will borrow from \cite{bonito2015numerical},
	and would require a substantial additional toolset.
	As such, it is here left as future work.

	The remainder of the current paper is structured as follows. 
	In section \ref{sec:preliminaries} we describe the notation used throughout the paper, 
	as well as give brief introductions to the theory of interpolation spaces 
	and some useful results in functional analysis. 
	Section \ref{sec:precond-negfraclap} is devoted to substantiating the above heuristic argument, 
	and show that provided we are given efficient preconditioners 
	for fractional $H(\div)$ operators with positive exponent, 
	we can construct efficient preconditioners for the fractional Laplacian with negative exponent. 
	Then, in section \ref{sec:decompositions} we propose such preconditioners as additive multigrid operators 
	and give sufficient conditions under which they are efficient. 
	Lastly, in section \ref{sec:numerical_experiments} we provide a series 
	of numerical experiments verifying the theoretical results obtained in this work.

	\section{Preliminaries}
	\label{sec:preliminaries}
	Let $\Omega$ be a bounded, polygonal domain in $\reals^n$, with boundary $\pd \Omega$. We denote by $L^2(\Omega)$ the space of square integrable functions on $\Omega$, with inner product $\inner{\cdot}{\cdot}$, and norm $\norm{\cdot}$. We denote by $H^1(\Omega)$ the usual Sobolev space of functions in $L^2(\Omega)$ with all first-order derivatives also in $L^2(\Omega)$. The closure of smooth functions with compact support in $\Omega$ we denote by $H_0^1(\Omega)$, and its dual space is $H^{-1}(\Omega)$. For $k \in \setof{-1,1}$, the inner product and norm of $H^k(\Omega)$ we denote by $\inner{\cdot}{\cdot}_k$ and $\norm{\cdot}_k$, respectively. Further, we let $H(\div; \Omega)$ denote the Hilbert space of square-integrable vector fields on $\Omega$ with square-integrable divergence, while we write $H(\curl; \Omega)$ to mean the space of square-integrable vector fields on $\Omega$ with square-integrable $\curl$. We let $\Lambda(\cdot,\cdot)$ denote the standard inner product on $H(\div; \Omega)$ defined by
	\eqns{
		\Lambda(\sigma, \tau) = \inner{\sigma}{\tau} + \inner{\div \sigma}{\div \tau}, \quad \sigma, \tau \in H(\div; \Omega).
	}
	
	In general, a Hilbert space $X$ is equipped with an inner product and norm, which we denote by $\inner{\cdot}{\cdot}_X$ and $\norm{\cdot}_X$, respectively, and its dual is denoted by $\dual{X}$. For two Hilbert spaces $X$ and $Y$, we write $\mathcal{L}(X,Y)$ to mean the space of bounded linear operators $T: X \to Y$, which we equip with the usual operator norm
	\eqns{
		\norm{T}_{\mathcal{L}(X,Y)} = \sup_{x \in X}\frac{\norm{Tx}_Y}{\norm{x}_X}.
	}
	
	Let now $A$ be a symmetric positive definite operator on a Hilbert space $X$.
	For sake of simplicity, we assume the spectrum of $A$ to be wholly discrete,
	i.e. $A$ has empty continuous- and residual spectrum.
	Denote by $\setof{(\lambda_k,\phi_k)}_{k=1}^\infty$ the set of eigenpairs of $A$, normalized so that
	\eqns{
		\inner{\phi_k}{\phi_l}_X = \delta_{k,l},
	}
	where $\delta_{k,l}$ is the Kronecker delta.  Then $\phi_k$, for $k=1,2,\ldots$ forms an orthonormal basis of $X$, and if $u \in X$ has the representation $u = \sum_{k=1}^\infty c_k \phi_k$, then
	
	\eqns{
		Au = \sum_{k=1}^\infty \lambda_k c_k \phi_k.
	}
	For $s\in \reals$, we define the fractional power $A^s$ of $A$ by
	\eqns{
		A^s u = \sum_{k=1}^\infty \lambda_k^s c_k \phi_k.
	}
	If $A$ is only positive semi-definite, then we must restrict to $s > 0$.
	If $B$ is another symmetric positive semi-definite operator on $X$, we write $A \leq B$ if for every $u \in X$
	\eqns{
		\inner{Au}{u}_X \leq \inner{Bu}{u}_X
	}
	holds. Note that $A \geq 0$ is equivalent to saying that $A$ is positive semi-definite. In addition, we shall write $A \leq 1$ to mean that $\inner{Au}{u}_X \leq \inner{u}{u}_X$ for every $u \in X$.
	
	A result in operator theory is the Löwner-Heinz inequality, which in our case states that if $A \leq B$, then
	\eqn{
		\label{eq:Loewner-Heinz}
		A^s \leq B^s, \quad s\in [0,1],
	}
	cf. for instance \cite{kato1952notes}. Inequality \eqref{eq:Loewner-Heinz} means that the function $x^s$ with $x\in [0,\infty)$ is operator monotone for $s\in[0,1]$.
	It follows that  $-(x)^s$ is operator convex (cf. \cite[Thm. 2.1 and 2.5]{hansen1982monotone}), that is, for any two symmetric positive semi-definite operators $A$ and $B$ on a Hilbert space $X$, the inequality
	\eqns{
		\label{eq:operator-convex-def}
		\lambda A^s + (1-\lambda)B^s \leq \left(\lambda A + (1-\lambda)B \right)^s
	}
	holds for every $\lambda \in [0,1]$.
	A key result regarding operator convex functions is the Jensen's operator inequality (cf. \cite[Theorem 2.1]{hansen2003jensen}). 
	The version we will use in the current work states that for any bounded, symmetric positive semi-definite operator $A$ on $X$, and $P: X \to X$ so that $\adjoint{P}P \leq 1$
	\eqn{
		\label{eq:jensen-inequality}
		\adjoint{P}A^s P \leq \left( \adjoint{P}A P \right)^s.
	}
	We will at numerous times in this paper be in a position where we want to use \eqref{eq:jensen-inequality}, but where $P$ is a contraction between different Hilbert spaces. Thus, we make the following slight generalization of \eqref{eq:jensen-inequality}.
	\begin{lemma}
		\label{lem:gen-jensen-inequality}
		Let $X_1$ and $X_2$ be two Hilbert spaces, and $T: X_1 \to X_2$ an operator satisfying $\adjoint{T}T \leq 1$ on $X_1$.
		Further, assume that $A$ is a bounded, symmetric positive semi-definite operator on $X_2$. Then 
		\eqn{
			\label{eq:gen-jensen-inequality}
			\adjoint{T}A^s T \leq \left( \adjoint{T} A T \right)^s
		} 
		for every $s \in [0,1]$.
	\end{lemma}
	\begin{proof}
		See that \eqref{eq:gen-jensen-inequality} holds for $s=0$ and $s=1$, so fix $s \in (0,1)$. We define the auxiliary Hilbert space $X = X_1 \oplus X_2$, with inner product inherited from the inner products on $X_1$ and $X_2$. Now, define linear operators $P$ and $\tilde{A}$ on $X$ as
		\eqns{
			P = \begin{pmatrix}
				0 & 0 \\
				T & 0 
			\end{pmatrix},
			\quad \text{ and } \quad
			\tilde{A} = \begin{pmatrix}
				0 & 0 \\
				0 & A
			\end{pmatrix}. 
		}
		A simple calculation then shows that
		\eqns{
			\adjoint{P}P = \begin{pmatrix}
				\adjoint{T}T & 0 \\
				0 & 0
			\end{pmatrix}
			\leq 1,
		}
		by the assumption on $T$. Similarly,
		\eqns{
			\adjoint{P}\tilde{A}^\theta P = \begin{pmatrix}
				\adjoint{T}A^\theta T & 0 \\
				0 & 0
			\end{pmatrix}
		}
		for every $ \theta > 0$. Then, we have from the standard Jensen's inequality in \eqref{eq:jensen-inequality} that
		\eqns{
			\begin{pmatrix}
				\adjoint{T}A^s T & 0 \\
				0 & 0 
			\end{pmatrix}
			= \adjoint{P}\tilde{A}^s P \leq \left( \adjoint{P}\tilde{A} P \right)^s
			= \begin{pmatrix}
				\left( \adjoint{T} A T \right)^s & 0 \\
				0 & 0
			\end{pmatrix}.
		}
		In particular, $\adjoint{T}A^s T \leq \left( \adjoint{T} A T \right)^s$, which completes the proof.
	\end{proof}
	
	\subsection{Interpolation spaces}
	In defining fractional Sobolev spaces and fractional $H(\div)$ spaces, we will use some results from interpolation theory, as presented in \cite{lions1972nonhom}, and so we shall make a quick review.
	
	Let $X$ and $Y$ be separable Hilbert spaces with inner products $\inner{\cdot}{\cdot}_X$ and $\inner{\cdot}{\cdot}_Y$, and corresponding norms $\norw{\cdot}{X}$ and $\norw{\cdot}{Y}$, respectively. 
	Furthermore, we assume that $X \subset Y$, with $X$ dense in $Y$ and continuous injection.
	In this case we call $X$ and $Y$ compatible.
	
	Denote by $D(A)$ the set of $u\in Y$ so that the linear form 
	\eqns{
		L_u(v) = \inner{u}{v}_X \quad v\in X
	}
	is continuous in $Y$. Following the discussion in \cite{lions1972nonhom}, we note that $D(A)$ is dense in $Y$.
	Using Riesz' representation theorem, there is a $w \in Y$ so that
	\eqns{
		\inner{w}{v}_Y = \inner{u}{v}_X.
	}
	The mapping $u \mapsto w$ defines an unbounded linear operator $A: D(A) \to Y$, which is defined by
	\eqn{
		\label{eq:A-def}
		\inner{Au}{v}_Y = \inner{u}{v}_X.
	}
	Clearly, $A$ is self-adjoint and positive. Using the spectral decomposition of self-adjoint operators, we may define the powers, $A^\theta$, $\theta \in \reals$, of $A$.
	We define interpolation spaces in the following way:
	\begin{mydef}
		\label{def:abstract-interpolation-space}
		Let $X$ and $Y$ satisfy the above assumptions. For $\theta \in [0,1]$ we define the interpolation space
		\eqn{
			\label{eq:intspace_def1}
			\interpolate{Y}{X}{\theta} = D(A^{\frac{\theta}{2}}) = \setof{u\in Y \sothat A^{\frac{\theta}{2}}u \in Y }
		}
		with norm given by the graph norm 
		\eqn{
			\label{eq:intspace_def2}
			\norm{u}_{\interpolate{Y}{X}{\theta}} := \left(\norw{u}{Y}^2 + \inner{A^\theta u}{u}_Y\right)^{\frac{1}{2}}.
		}
	\end{mydef}
	It follows by the definition that
	\eqns{
		\interpolate{Y}{X}{0} = Y, \text{ and } \interpolate{Y}{X}{1} = X.
	}
	The following is a key Theorem in interpolation theory.
	\begin{theorem}
		\label{thm:operator-interpolation}
		Let $\{X,Y\}$ and $\{\mathcal{X}, \mathcal{Y}\}$ be two pairs of compatible Hilbert spaces. 
		Further, let $T$ be a continuous operator $\mathcal{L}(X,\mathcal{X}) \cap \mathcal{L}(Y, \mathcal{Y})$, so that
		\algns{
			\norw{Tu}{\mathcal{X}} &\leq M_0\norw{u}{X}, \\
			\norw{Tu}{\mathcal{Y}} &\leq M_1\norw{u}{Y}.
		}
		Then $T \in \mathcal{L}(\interpolate{Y}{X}{\theta}, \interpolate{\mathcal{Y}}{\mathcal{X}}{\theta})$, and
		\eqn{
			\label{eq:operator-interpolation}
			\norw{Tu}{\interpolate{\mathcal{Y}}{\mathcal{X}}{\theta}} \leq CM_0^{1-\theta}M_1^{\theta}\norw{u}{\interpolate{Y}{X}{\theta}},
		}
		where $C$ is a constant independent of $T$, $\mathcal{X}$, and $\mathcal{Y}$.
	\end{theorem}
	
	If we now make the identification $Y = \dual{Y}$, then $Y \subset \dual{X}$ is dense, with continuous embedding. Thus,
	the interpolation space $\interpolate{\dual{X}}{Y}{\theta}$ is well-defined for $\theta \in [0,1]$ according to definition \ref{def:abstract-interpolation-space}. Moreover, we have that (cf. \cite[Thm. 6.2]{lions1972nonhom})
	\eqn{
		\label{eq:interpolation-duality}
		\interpolate{\dual{X}}{Y}{\theta} = \dual{\interpolate{Y}{X}{1-\theta}}.
	}
	
	It is well-known that $H^1(\Omega)$ is densely and continuously embedded in $L^2(\Omega)$, which implies that we can define the fractional Sobolev spaces $H^s(\Omega)$ for $s\in [0,1]$ as
	\eqns{
		H^s(\Omega) := \interpolate{L^2(\Omega)}{H^1(\Omega)}{s}
	}
	We go on to define $H^s_0(\Omega)$ as the closure in $H^s(\Omega)$ of smooth and compactly supported functions on $\Omega$, while for $s \in [-1,0]$, we define 
	\eqns{
		H^s(\Omega) = \dual{H^{-s}_0(\Omega)}
	}
	We note that this definition for negative fractional Sobolev spaces is equivalent to interpolation between $H^{-1}(\Omega)$ and $L^2(\Omega)$.
	
	Similarly, we define the fractional $H(\div;\Omega)$ space as
	\eqn{
		\label{eq:fractional-hdiv-def}
		H^s(\div; \Omega) := \interpolate{L^2(\Omega)}{H(\div; \Omega)}{s}.
	}

	\subsection{Discrete interpolation spaces}
	The discrete variant of fractional operators can be constructed analogously to the continuous setting. Suppose $X_h \subset X$ is a finite-dimensional subspace. We can define the operator $A_h: X_h \to X_h$ by
	\eqns{
		\inner{A_h v}{w}_Y = \inner{v}{w}_X.
	}
	We note that because $X_h$ is finite-dimensional, all norms are equivalent, and in particular, $A_h$ is a bounded operator.
	Since $A_h$ is SPD, we can define its fractional powers $A_h^\theta$ for $\theta \in \reals$, and discrete fractional norms $\norm{\cdot}_{\theta,h}^2 := \inner{A_h^\theta \cdot}{\cdot}$. When $\theta = 0$ and $\theta = 1$, the norm $\norm{\cdot}_{\theta,h}$ coincides with the $Y$- and $X$ norm, respectively. Furthermore, for $\theta \in (0,1)$ the discrete norm is equivalent to the $\interpolate{Y}{X}{\theta}$ norm, with constants of equivalence independent of $X_h$ (cf. \cite[Proposition 3.2]{arioli2009discrete})
	
	Suppose now that we have an additional finite-dimensional subspace $X_H \subset X_h$. Analogously to before we can define the SPD operator $A_H: X_H \to X_H$, and its fractional powers $A_H^\theta$, with $\theta \in \reals$. In the case of $\theta = 0$ or $\theta = 1$ we have that
	\eqns{
		\inner{A_H^\theta v}{w}_Y = \inner{A_h^\theta v}{w}_Y, \quad v,w \in X_H.
	}
	However, this inheritance of bilinear forms fails when $\theta \in (0,1)$. 
	Getting ahead of ourselves, the inheritance of bilinear forms is a common 
	assumption in the design and analysis of multigrid algorithms. 
	Therefore, that the inheritance fails to hold when $\theta \in (0,1)$ can be detrimental. 
	The following lemma shows that we are able to recover one of the key inequalities used in \cite{bramble1991analysis} in the analysis of multigrid algorithms on non-inherited bilinear forms.
	
	\begin{lemma}
		\label{lem:noninheritance-inequality}
		Let $\theta \in [0,1]$. We have that restricted to $X_H$
		\eqns{
			A^\theta_h \leq A^\theta_H.
		}
		That is, for every $v \in X_H$
		\eqn{
			\label{eq:abstract-noninheritance-inequality}
			\inner{A^\theta_h v}{v}_Y \leq \inner{A^\theta_H v}{v}_Y.
		}
	\end{lemma}
	\begin{proof}
		As already noted, for $\theta=0$ and $\theta=1$ \eqref{eq:abstract-noninheritance-inequality} holds with equality, so for the remainder of the proof let $0 < \theta < 1$.
		
		Let $I_H: X_H \to X_h$ be the inclusion operator, and $\adjoint{I_H}$ its adjoint with respect to the $Y$-inner product. Then, $\adjoint{I_H}I_H$ is the identity on $X_H$, so $\adjoint{I_H}I_H \leq 1$ holds trivially. 
		By Lemma \ref{lem:gen-jensen-inequality}, we thus have that
		\eqn{
			\label{eq:noninheritance-inequality-proof1}
			\adjoint{I_H}A_h^\theta I_H \leq \left( \adjoint{I_H} A_h I_H\right)^\theta.
		}
		The result follows from \eqref{eq:noninheritance-inequality-proof1} and the observation that $A_H = \adjoint{I_H}A_h I_H$.
	\end{proof}
	
	\section{Preconditioner for fractional Laplacian}
	\label{sec:precond-negfraclap}
	In this section we will establish a way to construct preconditioners for $(-\Delta)^{-s}$ when $s \in [0,1]$. We will begin by first considering the continuous setting, which will motivate the construction of preconditioners for a discretization of $(-\Delta)^{-s}$.
	We define $-\Delta: H_0^1(\Omega) \to H^{-1}(\Omega)$ by
	\eqns{
		\inner{(-\Delta) u}{v} = \inner{\grad u}{\grad v}, \quad u,v \in H_0^1(\Omega).
	}
	In view of the interpolation theory discussed in the previous section, 
	it is evident that $(-\Delta)^s$ is well-defined for any $s \in [0,1]$, 
	and it is an isomorphism from $H_0^s(\Omega)$ to $H^{-s}(\Omega)$. 
	We denote its inverse by $(-\Delta)^{-s}$, 
	and consider the problem of finding $u \in H^{-s}(\Omega)$ so that
	\eqn{
		\label{eq:continuous-fraclap-problem}
		(-\Delta)^{-s} u = f,
	}
	for a given $f \in H_0^s(\Omega)$. 
	To precondition \eqref{eq:continuous-fraclap-problem}, 
	we seek a self-adjoint isomorphism $B^s: H_0^s(\Omega) \to H^{-s}(\Omega)$, 
	so that 
	\eqn{
		\label{eq:negFracLap-cont-speceq}
		C_1 \norm{u}_{H^{-s}(\Omega)} \leq \inner{B^s u}{u} \leq C_2 \norm{u}_{H^{-s}(\Omega)}
	}
	for some constant $C_1, C_2 > 0$.
	
	Now, consider the gradient operator, $\grad$. It is clear that $\grad \in \mathcal{L}(H_0^1(\Omega),L^2(\Omega))$. On $L^2(\Omega)$, we define 
	\eqns{
		\inner{\grad u}{\tau} = - \inner{u}{\div \tau}, \quad u \in L^2(\Omega), \, \tau \in H(\div;\Omega).
	}
	Using integration by parts, this reduces to the standard $\grad$ when $u \in H_0^1(\Omega)$. 
	Moreover, we have that
	\eqns{
		\norm{\grad u}_{\dual{H(\div,\Omega)}} = \sup_{\tau \in H(\div; \Omega)}\frac{\inner{u}{\div \tau}}{\norm{\tau}_{H(\div;\Omega)}} \leq \norm{u}.
	}
	Thus,
	\eqns{
		\grad \in \mathcal{L}(H_0^1(\Omega),L^2(\Omega)) \cap \mathcal{L}(L^2(\Omega), \dual{H(\div; \Omega)}),
	}
	and Theorem \ref{thm:operator-interpolation} then implies that $\grad \in \mathcal{L}\left(H_0^s(\Omega), \interpolate{\dual{H(\div;\Omega)}}{L^2(\Omega)}{s}\right)$. In view of \eqref{eq:interpolation-duality} and \eqref{eq:fractional-hdiv-def} we can rewrite this as
	\eqn{
		\label{eq:grad-interpolation}
		\grad \in \mathcal{L}(H_0^s(\Omega), \dual{H^{1-s}(\div;\Omega)}). 
	}
	
	Suppose now that we are given a self-adjoint isomorphism $B_{\div}^{1-s}: \dual{H^{1-s}(\div; \Omega)} \to H^{1-s}(\div; \Omega)$ which for every $\tau \in \dual{H^{1-s}(\div; \Omega)}$ satisfies
	\eqn{
		\label{eq:cont-fracHdiv-speceq}
		C_{d,1} \norm{\tau}_{\dual{H^{1-s}(\div; \Omega)}}^2 \leq \inner{B_{\div}^{1-s} \tau}{\tau} \leq C_{d,2} \norm{\tau}_{\dual{H^{1-s}(\div; \Omega)}}^2
	}
	for some constants $C_{d,1}, C_{d,2} > 0$ independent of $\tau$. We then define
	\eqn{
		\label{eq:auxSpace-PC-cont-def}
		B^s = \adjoint{\grad} B^{1-s}_{\div} \grad.
	}
	
	Our aim is to show that $B^s$ defined by \eqref{eq:auxSpace-PC-cont-def} satisfies \eqref{eq:negFracLap-cont-speceq}.
	We begin by observing that $B^s$ is self-adjoint and maps elements from $H_0^s(\Omega)$ to $H^{-s}(\Omega)$.
	Moreover, the mapping property of $\grad$ in \eqref{eq:grad-interpolation} and the boundedness of $B^{1-s}_{\div}$
	imply that $B^s \in \mathcal{L}(H_0^s(\Omega), H^{-s}(\Omega))$.
	
	Establishing the lower bound of \eqref{eq:negFracLap-cont-speceq} is more difficult
	in that we want to interpolate between lower bounds on the gradient operator.
	However, Theorem \ref{thm:operator-interpolation} is not applicable in this setting.
	To overcome this problem, we will interpolate between bounds on a left-inverse, $T$, of $\grad$.
	In this work, we employ the Bogovski\u{\i} operator established in \cite{costabel2010bogovskii}.
	If $\Omega$ is star-shaped with respect to an open ball $B$, $T$ takes for a vector field $\tau$
	the explicit form
	\eqns{
		T\tau(x) = \int_{\Omega}K(x,y) (x-y)\cdot \tau(y) \intd y, \quad \text{ where } K(x,y) = \int_1^{\infty}(t-1)^{n-1}\theta(y+t(x-y))\intd t.
	}
	Here, $\theta \in \mathcal{C}_0^\infty(\reals^n)$ with support contained in $B$ and integrates to $1$.
	It can be checked that $T$ is a left-inverse of $\grad$, and satisfies
	\eqn{
		\label{eq:bogovskii-mapping-properties}
		T \in \mathcal{L}(L^2(\Omega), H_0^1(\Omega)) \cap \mathcal{L}(\dual{H(\div; \Omega)}, L^2(\Omega)),
	}
	see \cite[Cor. 3.4]{costabel2010bogovskii}.
	We note that the definition of $T$ can be extended to general Lipschitz domains ---
	as such domains are finite unions of star-shaped domains ---
	with the same mapping properties.
	From \eqref{eq:bogovskii-mapping-properties} and Theorem \ref{thm:operator-interpolation} we have that
	\eqn{
		\label{eq:bogovskii-interpolation}
		T \in \mathcal{L}(\dual{H^{1-s}(\div;\Omega)}, H_0^s(\Omega)).
	}
	
	Finally, we are in a position to prove that $B^s$ satisfies \eqref{eq:negFracLap-cont-speceq}, and hence is a suitable preconditioner for \eqref{eq:continuous-fraclap-problem}. 
	The result is stated in the following theorem.
	\begin{theorem}
		Let $s \in [0,1]$, and $B^{1-s}_{\div}: \dual{H^{1-s}(\div; \Omega)} \to H^{1-s}(\div; \Omega)$ satisfy \eqref{eq:cont-fracHdiv-speceq}. Then $B^s$ defined by \eqref{eq:auxSpace-PC-cont-def} satisfies \eqref{eq:negFracLap-cont-speceq} with
		\eqn{
			\label{eq:cont-negFracLap-constants}
			C_1 = C_{d,1}\norm{T}_{\mathcal{L}(\dual{H^{1-s}(\div;\Omega)}, H_0^s(\Omega))}^{-2} ,\quad \text{ and } \quad C_2 =  C_{d,2}\norm{\grad}_{\mathcal{L}(H_0^s(\Omega), \dual{H^{1-s}(\div;\Omega)})}^{2}.
		}
	\end{theorem}
	\begin{proof}
		Fix $s \in [0,1]$, and take any $u \in H_0^s(\Omega)$. From the definition of $B^s$, see that
		\eqns{
			\inner{B^s u}{u} = \inner{B^{1-s}_{\div} \grad u}{\grad u}. 
		}
		From the second inequality of \eqref{eq:cont-fracHdiv-speceq} and the mapping property of $\grad$ in \eqref{eq:grad-interpolation} we deduce that
		\eqns{
			\inner{B^s u}{u} \leq C_{d,2}\norm{\grad u}_{\dual{H^{1-s}(\div; \Omega)}}^2 \leq C_{d,2}\norm{\grad}_{\mathcal{L}(H_0^s(\Omega), \dual{H^{1-s}(\div;\Omega)})}^2\norm{u}_{H_0^s(\Omega)}^2,
		}
		which proves the second inequality of \eqref{eq:negFracLap-cont-speceq} with $C_2$ as given in \eqref{eq:cont-negFracLap-constants}.
		
		We can treat the lower bound of \eqref{eq:negFracLap-cont-speceq} similarly, 
		but now use the lower bound of \eqref{eq:cont-fracHdiv-speceq} and \eqref{eq:bogovskii-interpolation}.
		That is, we have 
		\eqn{
			\label{eq:cont-negFracLap-proof1}
			\inner{B^s u}{u} \geq C_{d,1}\norm{\grad u}^2_{\dual{H^{1-s}(\div; \Omega)}},
		}
		and, since $T\grad$ is the identity on $H_0^s(\Omega)$,
		\eqn{
			\label{eq:cont-negFracLap-proof2}
			\norm{u}_{H_0^s(\Omega)} = \norm{T\grad u}_{H_0^s(\Omega)} \leq \norm{T}_{\mathcal{L}(\dual{H^{1-s}(\div;\Omega)}, H_0^s(\Omega))} \norm{\grad u}_{\dual{H^{1-s}(\div; \Omega)}}.
		}
		Combining \eqref{eq:cont-negFracLap-proof1} and \eqref{eq:cont-negFracLap-proof2} yields
		\eqns{
			\inner{B^s u}{u}  \geq C_{d,1} \norm{T}_{\mathcal{L}(\dual{H^{1-s}(\div;\Omega)}, H_0^s(\Omega))}^{-2}\norm{u}_{H_0^s(\Omega)}^2.
		}
	\end{proof}
	\begin{rem}
		\label{rem:positive-fractional-Hdiv}
		With the definition of $B^s$ given in \eqref{eq:auxSpace-PC-cont-def}, we have essentially translated the problem of preconditioning $(-\Delta)^{-s}$ to the problem of preconditioning $\Lambda^{1-s}$. The advantage of this is that the latter problem has positive exponent, and so, as we will see, will have similar spectral properties to $\Lambda$, for which efficient preconditioning strategies have been studied earlier. 
	\end{rem}
	\subsection{Discrete setting}
	We will now use the construction of $B^s$ from the previous section as motivation to construct an analogous discrete operator.
	To that end, let $\T_h$ be a shape-regular triangulation of $\Omega$, with characteristic mesh size $h$. For $r \geq 0$, we let $S_h$ denote the space of all discontinuous, piecewise polynomials of degree at most $r$, subordinate to $\T_h$. That is,
	\eqns{
		S_h = \setof{u \in L^2(\Omega) \sothat u\big|_T \in P_r(T),\, \forall T \in \T_h}.
	}
	We further let $V_h = \mathcal{RT}_r(\T_h) \subset H(\div; \Omega)$ be the Raviart-Thomas space of index $r$, and $C_h = \mathcal{NE}_r(\T_h) \subset H(\curl; \Omega)$ the Nedelec space of first kind of index $r$, both relative to the triangulation $\T_h$.
	It is then well-known that $\curl(C_h) \subset V_h$, and $\div(V_h) \subset S_h$.
	We define the discrete gradient operator $\grad_h: S_h \to V_h$ by
	\eqn{
		\label{eq:discrete-grad-def}
		\inner{\grad_h u}{\tau} = - \inner{u}{\div \tau}, \quad u \in S_h, \, \tau \in V_h,
	}
	and discrete curl operator $\curl_h: V_h \to C_h$ by
	\eqn{
		\label{eq:discrete-curl-def}
		\inner{\curl_h \tau}{q} = \inner{\tau}{\curl q}, \quad \tau \in V_h, \, q \in C_h.
	}
	With these definitions, we have the discrete Helmholtz decomposition $V_h = \curl C_h \oplus \grad_h S_h$. That is, every $\tau \in V_h$ can be written as 
	\eqn{
		\label{eq:discrete-helmholtz}
		\tau = \grad_h u + \curl q,
	}
	for unique $u \in S_h$ and $q \in \curl_h V_h$. Cf. e.g. \cite{arnold2000multigrid}. Moreover, this decomposition is orthogonal in both $\inner{\cdot}{\cdot}$ and $\Lambda(\cdot,\cdot)$.
	
	To get a discrete analogue of the preconditioner $B^s$ in \eqref{eq:auxSpace-PC-cont-def}, we further need to define discrete counterparts to the operators $-\Delta$ and $\Lambda$. To that end, we define the discrete Laplacian as $A_h:= \adjoint{\grad_h}\grad_h$, i.e. $A_h$ is the symmetric operator on $S_h$ that satisfies
	\eqn{
		\label{eq:Ah-def}
		\inner{A_h u}{v} = \inner{\grad_h u}{\grad_h v}, \quad u,\, v\in S_h.
	}
	Lastly, since $V_h$ is a conforming discretization of $H(\div;\Omega)$, we simply take $\Lambda_h: V_h \to V_h$ to be the restriction of $\Lambda$ to $V_h$. In other words,
	\eqns{
		\inner{\Lambda_h \sigma}{\tau} = \Lambda\inner{\sigma}{\tau},\quad \sigma, \, \tau \in V_h.
	}
	
	It is well-known that (cf. for instance \cite{brezzi2013mixed}), with these particular choices of $S_h$ and $V_h$, there is a $\beta>0$ indepedent of $h$ so that for every $u \in S_h$
	\eqn{
		\label{eq:discrete-inf-sup}
		\sup_{\tau \in V_h} \frac{\inner{u}{\div \tau}}{\inner{\Lambda_h \tau}{\tau}^{\frac{1}{2}}} \geq \beta \norm{u}.
	}
	This implies that $\div: V_h \to S_h$ is surjective or, equivalently, that $\grad_h: S_h \to V_h$ is injective.
	As a consequence, $A_h$ is not only symmetric, but also positive-definite, and so $A_h^s$ is well-defined for every $s \in \reals$. The discrete counterpart to \eqref{eq:continuous-fraclap-problem} is then to find, for $s\in [0,1]$ and $f \in S_h$, a $u \in S_h$ such that
	\eqn{
		\label{eq:discrete-fraclap-problem}
		A_h^{-s}u = f.
	}
	To precondition \eqref{eq:discrete-fraclap-problem}, we seek a symmetric positive definite operator $B^s_h: S_h \to S_h$ which is easy to compute and spectrally equivalent to $A_h^s$, with constants of equivalence independent of $h$. Using the previous continuous preconditioner defined in \eqref{eq:auxSpace-PC-cont-def} as motivation, we will see that 
	\eqn{
		\label{eq:auxSpace-discrete-PC-def}
		B^s_h = \adjoint{\grad_h} B_{\div,h}^{1-s}\grad_h,
	}
	where $B_{\div,h}^{1-s}: V_h \to V_h$ is a symmetric positive definite operator spectrally equivalent to $\Lambda_h^{-(1-s)}$, leads to an efficient preconditioner for $A^{-s}_h$. The key result in this section is given in Theorem \ref{thm:auxSpace-discrete-PC-speceq} below, whose proof will resemble the argument we made in the continuous setting. In particular, we must ensure that $\grad_h$ has the appropriate upper and lower bounds when $s=0$ and $s=1$. As we will see, the intermediate cases will then follow from Jensen's operator inequality.
	
	For the upper bounds of $\grad_h$, we have from the definitions of $\grad_h$ and $\Lambda_h$ that
	\eqn{
		\label{eq:gradh-upper-bound}
		\algnd{
			\inner{\Lambda_h^{-1}\grad_h u}{\grad_h u} &= \norm{\Lambda_h^{-\frac{1}{2}}\grad_h u}^2
			= \sup_{\tau \in V_h}\frac{\inner{\Lambda_h^{-\frac{1}{2}}\grad_h u}{\tau}^2}{\norm{\tau}^2} \\
			&= \sup_{\tau \in V_h} \frac{\inner{\grad_h u}{\tau}^2}{\inner{\Lambda_h \tau}{\tau}}
			\leq \norm{u}^2,
		}
	}
	which is the discrete analogue to $\grad \in \mathcal{L}(L^2(\Omega), \dual{H(\div;\Omega)})$. The discrete analogue to $\grad \in \mathcal{L}(H_0^1(\Omega), L^2(\Omega) )$ is simply that $\norm{\grad_h u}^2 = \inner{A_h u}{u}$.
	
	For the necessary lower bounds on $\grad_h$, we define $L: V_h \to S_h$ by $L \tau = u$ according to the discrete Helmholtz decomposition \eqref{eq:discrete-helmholtz}. It is then evident that $L\grad_h$ is the identity on $S_h$. That $L$ satisfies the discrete analogues to
	\eqref{eq:bogovskii-mapping-properties} is given in the following lemma.
	\begin{lemma}
		\label{lem:discrete-lifting-properties}
		With $L: V_h \to S_h$ as defined above, it holds for every $\tau \in V_h$ that
		\eqn{
			\label{eq:discrete-lifting-properties}
			\norm{L\tau}^2 \leq \beta^{-2} \inner{\Lambda_h^{-1}\tau}{\tau}, \quad \text{ and } \quad \inner{A_h L \tau}{L \tau } \leq \norm{\tau}^2,
		}
		where $\beta$ is given by \eqref{eq:discrete-inf-sup}.
	\end{lemma}
	\begin{proof}
		Fix $\tau \in V_h$, and let $u = L\tau$. From \eqref{eq:discrete-inf-sup} and the decomposition \eqref{eq:discrete-helmholtz}, we have that
		\eqns{
			\beta\norm{u} \leq \sup_{\sigma \in V_h} \frac{\inner{\grad_h u}{\sigma}}{\inner{\Lambda_h \sigma}{\sigma}^{\frac{1}{2}}} \leq \sup_{\sigma \in V_h} \frac{\inner{\tau}{\sigma}}{\inner{\Lambda_h \sigma}{\sigma}^{\frac{1}{2}}}.
		}
		Replacing $\sigma$ by $\Lambda_h^{-\frac{1}{2}}\sigma$ in the above yields
		\eqns{
			\beta\norm{u} \leq \sup_{\sigma \in V_h}\frac{\inner{\Lambda_h^{-\frac{1}{2}}\tau}{\sigma}}{\norm{\sigma}} \leq \inner{\Lambda_h^{-1}\tau}{\tau}^{\frac{1}{2}},
		}
		which proves the first inequality of \eqref{eq:discrete-lifting-properties}.
		
		The definitions of $A_h$ and $L$, and the $L^2$-orthogonality of the decomposition \eqref{eq:discrete-helmholtz} imply the second inequality of \eqref{eq:discrete-lifting-properties}, since
		\eqns{
			\inner{A_h L\tau}{L\tau} = \norm{\grad_h u}^2 \leq \norm{\tau}^2.
		}
	\end{proof}
	
	We are now in a position to state and prove the main spectral equivalence result of this section, from which the spectral equivalence between $B^s_h$ given in \eqref{eq:auxSpace-discrete-PC-def} and $A^s_h$ will readily follow.
	\begin{theorem}
		\label{thm:auxSpace-discrete-PC-speceq}
		Let $\grad_h$, $A_h$ and $\Lambda_h$ be defined as above, and let $s\in [0,1]$. Then, for every $u \in S_h$
		\eqn{
			\label{eq:discrete-auxSpace-speceq}
			\beta^{2(1-s)} \inner{A_h^{s} u}{u} \leq \inner{\Lambda_h^{-(1-s)}\grad_h u}{\grad_h u} \leq \inner{A_h^s u}{u},
		}
		where $\beta$ is given by \eqref{eq:discrete-inf-sup}.
	\end{theorem}
	\begin{proof}
		Fix $u\in S_h$ and $s\in [0,1]$.
		We begin by proving the second inequality of \eqref{eq:discrete-auxSpace-speceq}.
		Define $T_1 = \Lambda_h^{-\frac{1}{2}}\grad_h: S_h \to V_h$. From  \eqref{eq:gradh-upper-bound} it follows that $\adjoint{T_1}T_1 \leq 1$. Thus, Lemma \ref{lem:gen-jensen-inequality} implies that
		\eqn{
			\label{eq:discrete-auxspace-proof1}
			\adjoint{T_1}\Lambda_h^s T_1 \leq \left( \adjoint{T_1}\Lambda_h T_1 \right)^s.
		}
		Inserting the definition of $T_1$ into \eqref{eq:discrete-auxspace-proof1} yields
		\eqns{
			\adjoint{\grad_h}\Lambda_h^{-(1-s)}\grad_h \leq \left(\adjoint{\grad_h} \grad_h \right)^s = A_h^s,
		}
		which is equivalent to the second inequality of \eqref{eq:discrete-auxSpace-speceq}.	
		
		In proving the first inequality of \eqref{eq:discrete-auxSpace-speceq}, we will again make use of Lemma \ref{lem:gen-jensen-inequality}. To that end, we now set $T_2 = \beta L \Lambda_h^{\frac{1}{2}}$, and from Lemma \ref{lem:discrete-lifting-properties} it follows that $\adjoint{T_2}T_2 \leq 1$. Thus, an application of Lemma \ref{lem:gen-jensen-inequality} yields
		\eqns{
			\adjoint{T_2}A_h^s T_2 \leq \left( \adjoint{T_2} A_h T_2 \right)^s,
		}	
		which after inserting the definition of $T_2$ becomes
		\eqn{
			\label{eq:discrete-auxspace-proof2}
			\beta^{2(1-s)} \Lambda_h^{\frac{1}{2}}\adjoint{L} A_h^s L \Lambda_h^{\frac{1}{2}} \leq \left( \Lambda_h^{\frac{1}{2}}\adjoint{L} A_h L \Lambda_h^{\frac{1}{2}} \right)^s.
		}
		
		From Lemma \ref{lem:discrete-lifting-properties} $\adjoint{L}A_h L \leq 1$. Pre- and post multplying this inequality by $\Lambda_h^{\frac{1}{2}}$ and using the Löwner-Heinz inequality \eqref{eq:Loewner-Heinz}, we deduce that
		\eqn{
			\label{eq:discrete-auxspace-proof3}
			\left( \Lambda_h^{\frac{1}{2}}\adjoint{L} A_h L \Lambda_h^{\frac{1}{2}} \right)^s \leq \Lambda_h^s.
		}
		We now use \eqref{eq:discrete-auxspace-proof2} together with \eqref{eq:discrete-auxspace-proof3} and pre- and post multiply by $\Lambda_h^{-\frac{1}{2}}$ to get
		\eqns{
			\beta^{2(1-s)}\adjoint{L}A_h^s L \leq \Lambda_h^{-(1-s)}.
		}
		Finally, multiplying from the left by $\adjoint{\grad_h}$ and from the right by $\grad_h$, and using that both $\adjoint{\grad_h}\adjoint{L}$ and $L \grad_h$ are the identity on $S_h$, we arrive at
		\eqns{
			\beta^{2(1-s)}A^s_h \leq \adjoint{\grad_h}\Lambda_h^{-(1-s)}\grad_h,
		}
		which is the first inequality of \eqref{eq:discrete-auxSpace-speceq}.
	\end{proof}
	
	\begin{cor}
		\label{cor:discrete-auxSpace-speceq}
		Under the same assumptions as in Theorem \ref{thm:auxSpace-discrete-PC-speceq}, suppose we are given a symmetric positive definite operator $B_{\div,h}^{1-s}: V_h \to V_h$ spectrally equivalent to $\Lambda_h^{-(1-s)}$. That is, there are constants $C_1, C_2 > 0$ so that
		\eqn{
			\label{eq:discrete-fracHdiv-speceq}
			C_1 \inner{\Lambda_h^{-(1-s)}\tau}{\tau} \leq \inner{B_{\div,h}^{1-s} \tau}{\tau} \leq C_2\inner{\Lambda_h^{-(1-s)}\tau}{\tau}
		}
		for every $\tau \in V_h$. Then $B^s_h$ defined by \eqref{eq:auxSpace-discrete-PC-def} satisfies
		\eqn{
			\label{eq:discrete-Bs-speceq}
			C_1 \beta^{2(1-s)}\inner{A_h^s u}{u} \leq \inner{B^s_h u}{u} \leq C_2 \inner{A^s_h u}{u}
		}
		for every $u \in S_h$.
	\end{cor}
	
	\begin{proof}
		Take any $s\in [0,1]$ and $u \in S_h$. By the definition of $B^s_h$, the second inequalities of \eqref{eq:discrete-fracHdiv-speceq} and \eqref{eq:discrete-auxSpace-speceq}
		\eqns{
			\inner{B^s_h u}{u} \leq C_2\inner{\Lambda_h^{-(1-s)}\grad_h u}{\grad_h u} \leq C_2\inner{A^s_h u}{u},
		}
		which proves the second inequality of \eqref{eq:discrete-Bs-speceq}. The first inequality is proved similarly, using the lower bounds in \eqref{eq:discrete-fracHdiv-speceq} and \eqref{eq:discrete-auxSpace-speceq}.
	\end{proof}
	\begin{rem}
		\label{rem:implementational-concerns}
		At this point it is worth remarking on the implementation of $B^s_h$. 
		In computer code, a function $u \in S_h$ can have two distinct 
		representations as vectors in $\reals^{N_S}$, where $N_S = \dim S_h$. 
		Let $\setof{\phi_h^i}_{i=1}^{N_S}$ be a basis for $S_h$. 
		Then, if $u = \sum_{i=1}^{N_S} c_i \phi_h^i$, we call the vector 
		$\vec{u} = (c_1,\ldots, c_{N_S})^T \in \reals^{N_S}$ the coefficient vector representation of $u$, 
		while the vector $\dualvec{u} \in \reals^{N_S}$ with entries $\dualvec{u}_i = \inner{u}{\phi_h^i}$, 
		the dual vector representation of $u$. 
		Cf. e.g. \cite[Sec. 15]{bramble1993multigrid} for more details. 
		Let $\setof{\psi_h^i}_{i=1}^{N_V}$, with $N_V = \dim V_h$, be a basis for $V_h$. 
		For $\tau \in V_h$, let $\vec{\tau}$ and $\dualvec{\tau}$ be the analogous 
		coefficient- and dual vector representations of $\tau$. 
		The most straightforward matrix realization of $\grad_h$ is then 
		the matrix $\mat{D}_h \in \reals^{N_V \times N_S}$ with entries
		\eqns{
			(\mat{D}_h)_{i,j} = -\inner{\phi_h^i}{\div \psi_h^i}.
		}
		We see that $\mat{D}_h$ takes coefficient vectors in $\reals^{N_S}$ 
		and returns dual vectors in $\reals^{N_V}$. 
		Conversely, the transpose $\mat{D}_h^T$ takes coefficient vectors in $\reals^{N_V}$ 
		as input and returns dual vectors in $\reals^{N_S}$. 
		If $\mat{B}^{1-s}_{\div,h}$ is the matrix realization of $B^{1-s}_{\div,h}$ 
		taking dual vectors as input and returning coefficient vectors, 
		$B^s_h$ can be realized by the matrix
		\eqns{
			\mat{B}^s_h = \mat{D}_h^T \mat{B}^{1-s}_{\div, h} \mat{D}_h.
		}
		Then, $\mat{B}^s_h$ takes coefficient vectors as input and returns dual vectors, 
		which is opposite to usual implementations of preconditioners. 
		Thus, if this preconditioner should be used as part of a preconditioner 
		for problems of the form \eqref{eq:trace-prob-positive}, some care is needed. 
		In particular, the Lagrange multiplier $\lambda$ should be represented as a dual vector, 
		while the trace constraint $g$ should be represented by a coefficient vector. 
		We see then that the matrix realization of the trace operator $T$ 
		should take coefficient vectors to coefficient vectors. 
		That is, the matrix is simply a mapping of degrees of freedom from one space to another, 
		and no numerical integration is needed.
	\end{rem}
	By Corollary \ref{cor:discrete-auxSpace-speceq}, 
	we know that we can construct an efficient preconditioner for $A_h^{-s}$, 
	provided we have an efficient preconditioner for $\Lambda_h^{1-s}$ at our disposable. 
	This is by no means a given. 
	However, we will in the next section propose a construction of $B^{1-s}_{\div,h}$ on $V_h$ 
	satisfying \eqref{eq:discrete-fracHdiv-speceq} based on an additive multigrid approach.
	
	\section{Additive multigrid methods for $\Lambda_h^{s}$}
	\label{sec:decompositions}
	Recall that in section \ref{sec:precond-negfraclap} we constructed an efficient preconditioner for $A_h^{-s}$, where $A_h$ is a discrete Laplacian on $S_h$ and $s \in [0,1]$ provided we are given an efficient preconditioner for $\Lambda_h^{1-s}$ on $V_h$, which we denote by $B^{1-s}_{\div,h}$. In this section we give one construction of $B^{1-s}_{\div,h}$ based on a multigrid approach similar to that presented in \cite{baerland2018multigrid}. 
	
	To motivate the construction we note that multigrid methods, and other space decomposition methods, are popular and well-studied preconditioning strategies for $H(\div)$ problem.
	A key observation is that $\Lambda_h$ reduces to the identity operator on the kernel of $\div$ in $V_h$, 
	while on the $L^2$-orthogonal complement $\Lambda_h$ roughly behaves like an elliptic operator with a zero-order term. 
	In particular, $\Lambda_h$ can be decomposed into operators where subspace decomposition methods 
	have proven to be efficient.  
	We will now see that this line of reasoning continues to hold for $\Lambda_h^{s}$.
	To that end, consider the discrete Helmholtz decomposition of $v \in V_h$ given in \eqref{eq:discrete-helmholtz},
	\eqn{
		\label{eq:discrete-helmholtz2}
		\tau = \grad_h u + \curl q,
	}
	where $u \in S_h$ and $q \in \curl_hV_h$. From the definition of $\grad_h$, we have that $\Lambda_h = I + \grad_h \adjoint{\grad_h}$, which when applied to \eqref{eq:discrete-helmholtz2} yields
	\eqn{
		\label{eq:Lambda_h-helmholtz-decomposition}
		\Lambda_h \tau = \grad_h (I + A_h)u + \curl q,
	}
	where we recall that $A_h = \adjoint{\grad_h}\grad_h$ is a discrete Laplacian. 
	We see that $\Lambda_h$ is invariant in both $\grad_h S_h$ and its orthogonal complement, $\curl C_h$. 
	From \eqref{eq:Lambda_h-helmholtz-decomposition} it is also evident that the projections 
	$\tau \mapsto \grad_h u$ and $\tau \mapsto \curl q$ both commute with $\Lambda_h$. 
	In accordance with the discussion made in \cite{davis1957schwarz}, 
	it follows that $\Lambda_h^{s}$ also leave the decomposition in \eqref{eq:discrete-helmholtz2} invariant. 
	Thus, $\Lambda_h^{s}$ reduces to the identity operator on $\curl C_h$, 
	and behaves like $(I + A_h)^{s}$ on $\grad_h S_h$. 
	Multigrid methods were shown to be computationally effective for such operators in \cite{baerland2018multigrid}, 
	and this motivates using a similar approach for constructing preconditioners for $\Lambda_h^{s}$.
	
	Before proceeding, some issues need to be adressed. 
	As shown in Lemma \ref{lem:noninheritance-inequality}, 
	the operators on each level will not be inherited. 
	Therefore, the analysis will follow the framework of \cite{bramble1991analysis}. 
	Another problem is that the computation of $\Lambda_h^s$ 
	requires solving a potentially large eigenvalue problem, 
	which can be prohibitively expensive. 
	As a consequence, we cannot assume that we can compute errors on each level. 
	Standard multigrid algorithms, such as V-cycle, should then be excluded. 
	For this reason, we design the operators as additive multigrid operators, \cite{bramble1990parallel}, 
	where the residual of the problem is transferred to every grid level, 
	and no application of $\Lambda_h^s$ is required.
	
	In the following, we will use the same multilevel decomposition as was used in \cite{arnold1997preconditioning}, 
	but we emphasize that the analysis extends to other decompositions, 
	such as that given in \cite{hiptmair1997multigrid}.
	
	To construct our multigrid operator for $\Lambda_h^{s}$ suppose $\T_h$ is the result of successive refinements. That is, we are given a sequence
	\eqns{
		\T_1 \subset \cdots \subset \T_J = \T_h,
	}
	of shape-regular triangulations of $\Omega$, and $\T_k$ has charachteristic mesh size $h_k$ for $k= 1, \ldots, J$. We will assume that the refinements are bounded, in the sense that there is a constant $\gamma \geq 1$ so that $h_{k-1} \leq \gamma h_k$ for $k=2,\ldots,J$. We note that in applications $\gamma$ is around $2$. For each $k$, we set $V_k = \mathcal{RT}_r(\T_k)$ as the Raviart-Thomas space of index $r$ relative to the mesh $\T_k$. We further define $S_k \subset S_h$ and $C_k \subset C_h$ analogously, as well as operators $\grad_k: S_k \to V_k$ and $\curl_k: V_k \to C_k$ as the $L^2$-adjoint of $\div$ and $\curl$, respectively.
	
	For each $k$, we define $\Lambda_k: V_k \to V_k$ by
	\eqns{
		\inner{\Lambda_k \sigma}{\tau} = \Lambda\inner{\sigma}{\tau}, \quad \sigma,\, \tau \in V_k.
	}
	It is evident that $\Lambda_k$ is symmetric positive-definite, and so $\Lambda_k^\theta$ is well-defined for every $\theta \in \reals$, and as a consequence of Lemma  \ref{lem:noninheritance-inequality}
	\eqn{
		\label{eq:noninheritance-Hdiv}
		\inner{\Lambda_k^s \tau}{\tau} \leq \inner{\Lambda_{k-1}^s \tau}{\tau}
	}
	for $s \in  [0,1]$ and $\tau \in V_{k-1}$. For every $k$ we define $Q_k: V \to V_k$ as the $L^2$-orthogonal projection and $P^s_{k,k-1}: V_k \to V_{k-1}$ by
	\eqns{
		\inner{\Lambda_{k-1}^s P^s_{k,k-1} \sigma_k}{\tau_{k-1}} = \inner{\Lambda_{k}^s \sigma_k}{\tau_{k-1}}, \quad \sigma_k \in V_k, \, \tau_{k-1} \in V_{k-1},
	}
	with the interpretation that $P_{1,0}^s = 0$. We go on to define $P^s_k := P^s_{k+1,k} \cdots P^s_{J,J-1}: V_h \to V_k$, which satisfies
	\eqns{
		\inner{\Lambda_k^s P^s_k \sigma}{\tau_k} = \inner{\Lambda_h^s \sigma}{\tau_k},
	}
	for every $\sigma \in V_h$ and $\tau_k \in V_k$.
	
	It follows by the definitions of $P^s_k$ and $Q_k$ that
	\eqn{
		\label{eq:AP=QA-Hdiv}
		\Lambda_k^s P^s_k = Q_k \Lambda_h^s.
	}
	Note that in general $P^s_k$ is not a projection, except when $s=0$ (in which case it coincides with $Q_k$) and $s=1$. However, when $s\in (0,1)$ we have for any $\tau \in V_k$ that
	\eqns{
		\algnd{
			\inner{\Lambda_k^s P^s_k \tau}{P^s_k \tau} &= \inner{\Lambda_h^s \tau}{P^s_k \tau} \\
			&\leq \inner{\Lambda_h^s \tau}{\tau}^{\frac{1}{2}} \inner{\Lambda_h^s P^s_k \tau}{P^s_k \tau}^{\frac{1}{2}}.
		}
	}
	Applying \eqref{eq:noninheritance-Hdiv} in the above, we deduce that
	\eqn{
		\label{eq:noninheritance-Hdiv-adjoint}
		\inner{\Lambda_k^s P^s_k \tau}{P^s_k \tau} \leq \inner{\Lambda_h^s \tau}{\tau}.
	}
	
	Suppose now that on each level $k$ we are given symmetric positive definite operators $R^s_k: V_k \to V_k$. 
	As is usual, we call these operators smoothers, and they should, in a sense to made clearer below, approximate $\Lambda_k^{-s}$. 
	We then define our additive multigrid preconditioner $B^s_{\div,h}: V_h \to V_h$ as
	\eqn{
		\label{eq:Bsdiv-def}
		B^s_{\div,h} = \sum_{k=1}^J R^s_k Q_k.
	}
	
	The following theorem gives sufficient conditions on the smoothers to establish spectral equivalence between $B^s_{\div,h}$ and $\Lambda_h^{-s}$. The proof will mostly follow by standard techniques, but some care is needed since the operators are not inherited between grid levels.
	\begin{theorem}
		\label{thm:Bsdiv-spectral-equivalence}
		Let $s\in [0,1]$ and suppose that for each $k=1,\ldots,J$, the operator $R^s_k$ as defined above satisfies for every $\tau \in V_k$
		\eqn{
			\label{eq:Rs-condition1}
			\inner{R^s_k \tau}{\tau} \leq C_1 \inner{\Lambda_k^{-s}\tau}{\tau},
		}
		and
		\eqn{
			\label{eq:Rs-condition2}
			\inner{\left(R^s_k\right)^{-1}(I-P^s_{k,k-1})\tau}{(I-P^s_{k,k-1})\tau} \leq C_2 \inner{\Lambda_k^{s}(I-P^s_{k,k-1})\tau}{(I-P^s_{k,k-1})\tau}
		}
		for some constants $C_1$ and $C_2$ that are independent of $k$. Then,
		\eqn{
			\label{eq:Bsdiv-spectral-equivalence}
			C_2^{-1} \inner{\Lambda_h^{s} \tau}{\tau} \leq \inner{B^s_{\div,h}\Lambda_h^s \tau}{\Lambda_h^s \tau} \leq C_1 J \inner{\Lambda_h^{s}\tau}{\tau}
		}
	\end{theorem}
	\begin{proof}
		For the upper bound of \eqref{eq:Bsdiv-spectral-equivalence}, straightforward application of the definitions of $B^s_{\div,h}$ and $P^s_k$ show that
		\eqn{
			\label{eq:Bsdiv-spectral-equivalence-proof1}
			\inner{B^s_{\div,h}\Lambda_h^{s}\tau}{\Lambda_h^s \tau} = \sum_{k=1}^J \inner{R^s_k \Lambda_k^s P^s_k \tau}{\Lambda_k^s P^s_k \tau}.
		}
		Assumption \eqref{eq:Rs-condition1} and the non-inheritance inequality \eqref{eq:noninheritance-Hdiv-adjoint} then imply that
		\eqns{
			\inner{B^s_{\div,h}\Lambda_h^{s}\tau}{\Lambda_h^s \tau} \leq C_1 \sum_{k=1}^J \inner{\Lambda_k^s P^s_k \tau}{P^s_k \tau} \leq C_1 J \inner{\Lambda_h^s \tau}{\tau}.
		}
		
		In proving the lower bound of \eqref{eq:Bsdiv-spectral-equivalence}, we consider the decomposition $\tau = \sum_{k=1}\tau_k$, with $\tau_k = (P^s_k - P^s_{k-1})\tau = (I - P^s_{k,k-1})P^s_k \tau \in V_k$, for $k=1,\ldots,J$. Here, we interpret $P_0 = 0$ and $P_J = I$. Then,
		\eqns{
			\inner{\Lambda_h^s \tau}{\tau} = \sum_{k=1}^J \inner{\Lambda_k^s P^s_k \tau}{\tau_k} = \sum_{k=1}^J \inner{R^s_k \Lambda_k^s P^s_k \tau}{\left( R^s_k \right)^{-1}\tau_k}.
		}
		Since for every $k$, $R^s_k$ is symmetric positive definite, we can use Cauchy-Schwarz' and assumption \eqref{eq:Rs-condition2}, resulting in
		\eqn{
			\label{eq:Bsdiv-spectral-equivalence-proof2}
			\algnd{
				\inner{\Lambda_h^s \tau}{\tau} &\leq \sum_{k=1}^J \inner{R^s_k \Lambda_k^s P^s_k \tau}{\Lambda_k^s P^s_k \tau}^{\frac{1}{2}} \inner{\left( R^s_k \right)^{-1}\tau_k}{\tau_k}^{\frac{1}{2}} \\
				&\leq \left( \sum_{k=1}^J  \inner{R^s_k \Lambda_k^s P^s_k \tau}{\Lambda_k^s P^s_k \tau} \right)^{\frac{1}{2}}\left(C_2\sum_{k=1}^J \inner{\Lambda_k^s \tau_k}{\tau_k} \right)^{\frac{1}{2}} \\
				&= \inner{B^s_{\div,h}\Lambda_h^s \tau}{\Lambda_h^s \tau}^{\frac{1}{2}}\left(C_2\sum_{k=1}^J \inner{\Lambda_k^s \tau_k}{\tau_k} \right)^{\frac{1}{2}},
			}
		}
		where in the last step we have used \eqref{eq:Bsdiv-spectral-equivalence-proof1}. In view of \eqref{eq:Bsdiv-spectral-equivalence-proof2}, it only remains to show that
		\eqn{
			\label{eq:Bsdiv-spectral-equivalence-proof3}
			\sum_{k=1}^J \inner{\Lambda_k^s \tau_k}{\tau_k} \leq \inner{\Lambda_h^s \tau}{\tau}
		}
		to prove the lower bound of \eqref{eq:Bsdiv-spectral-equivalence}. Inserting the definition of $\tau_k$ and expanding factors, we find that
		\eqns{
			\inner{\Lambda_k^s \tau_k}{\tau_k} = \inner{\Lambda_k^s P^s_{k}\tau}{P^s_{k}\tau} - 2\inner{\Lambda_k^s P^s_k \tau}{P^s_{k-1}\tau} + \inner{\Lambda_k^s P^s_{k-1} \tau}{P^s_{k-1}\tau}. 
		}
		For the second term on the right hand side in the above, 
		we have $\inner{\Lambda_k^s P^s_k \tau}{P^s_{k-1}\tau} = \inner{\Lambda^s_{k-1}P^s_{k-1}\tau}{P^s_{k-1}\tau}$
		since $P^s_{k-1} = P^s_{k,k-1}P^s_k$, 
		while for the third term we apply \eqref{eq:noninheritance-Hdiv}. 
		Thus,
		\eqns{
			\inner{\Lambda_k^s \tau_k}{\tau_k} \leq \inner{\Lambda_k^s P^s_k \tau}{P^s_k \tau} - \inner{\Lambda_{k-1}^s P^s_{k-1}\tau}{P^s_{k-1}\tau}.
		}
		It follows that
		\eqns{
			\sum_{k=1}\inner{\Lambda_k^s \tau_k}{\tau_k} \leq \sum_{k=1}^J \left[ \inner{\Lambda_k^s P^s_k \tau}{P^s_k \tau} - \inner{\Lambda_{k-1}^s P^s_{k-1}\tau}{P^s_{k-1}\tau}\right] \leq \inner{\Lambda_h^s \tau}{\tau}.
		}
	\end{proof}
	
	Now it remains to choose smoothers satisfying the assumptions in Theorem \ref{thm:Bsdiv-spectral-equivalence}, and in this work we consider additive Schwarz operators based on the same space decomposition as in \cite{arnold1997preconditioning}. 
	
	For $k\geq 2$, let $\N_k$ denote the set of vertices in $\T_k$, and for each $\nu \in \N_k$, let $\T_{k,\nu}$ be the set of simplices meeting at the vertex $\nu$. Then $\T_{k,\nu}$ forms a triangulation of a small subdomain $\Omega_{k,\nu}$, and we define $V_{k,\nu}$ to be the subspace of functions in $V_k$ with support contained in $\bar{\Omega}_{k,\nu}$. The operators $\Lambda_{k,\nu}: V_{k,\nu} \to V_{k,\nu}$ and $P^s_{k,\nu}, \, Q_{k,\nu}: V_k \to V_{k,\nu}$ are then defined analogously to the corresponding operators above. We then define
	\eqn{
		\label{eq:Rs-def}
		R^s_k = \sum_{\nu \in \N_k} \Lambda_{k,\nu}^{-s}Q_{k,\nu},
	}
	while on the coarsest level we set $R^s_1 = \Lambda_1^{-s}$. It is well-known that additive Schwarz operator of the form \eqref{eq:Rs-def} are symmetric positive definite, and its inverse satisfies for $\tau \in V_k$
	\eqn{
		\label{eq:Rs-inv-identity}
		\inner{\left(R^s_k\right)^{-1} \tau}{\tau} = \inf_{\underset{\tau_\nu \in V_{k,\nu}}{\tau = \sum_{\nu} \tau_\nu}} \sum_{\nu \in \N_k} \inner{\Lambda_{k,\nu}^s  \tau_\nu}{\tau_\nu}.
	}
	Moreover, the decomposition $V_k = \sum_{\nu \in \N_k} V_{k,\nu}$ is $L^2$-stable in the sense that for every $\tau \in V_k$ there are $\tau_\nu \in V_{k,\nu}$ so that $\tau = \sum_{\nu} \tau_\nu$ and
	\eqn{
		\label{eq:L2-stable-decomposition}
		\sum_{\nu \in \N_k} \norm{\tau_\nu}^2 \leq c \norm{\tau}^2,
	}
	for some constant $c$, independent of $k$ and $v$. The analogue to \eqref{eq:L2-stable-decomposition} continues to hold if we define the decomposition $C_k = \sum_{\nu}C_{k,\nu}$ similarly (cf. \cite{arnold2000multigrid}).
	
	The verification of Assumption \eqref{eq:Rs-condition1} in Theorem \ref{thm:Bsdiv-spectral-equivalence} is given in the following lemma.
	\begin{lemma}
		\label{lem:Rs_k-upper-bound}
		For $k=1,\ldots,J$ and $s\in [0,1]$, let $R^s_k$ defined as above. Then there are constants $K_0, K_1 \geq 0$, independent of $k$ so that
		\eqn{
			\label{eq:R_k-upper-bound-01}
			\algnd{
				\inner{R^0_k \tau}{\tau} &\leq K_0\inner{\tau}{\tau} \\
				\inner{R^1_k \tau}{\tau} &\leq K_1 \inner{\Lambda_k^{-1}\tau}{\tau}
			}
		}
		for every $\tau\in V_k$. Moreover, for $s \in (0,1)$ and every $\tau \in V_k$
		\eqn{
			\label{eq:R_k-upper-bound-s}
			\inner{R^s_k \tau}{\tau} \leq K_0^{1-s}K_1^s \inner{\Lambda_{k}^{-s}\tau}{\tau},
		}
		where $K_0$ and $K_1$ are the same as in \eqref{eq:R_k-upper-bound-01}.
	\end{lemma}
	\begin{proof}
		The assertions are evident when $k=1$, with $K_0 = K_1 = 1$, so let $k \geq 2$. A proof of the second inequality of \eqref{eq:R_k-upper-bound-01} can be found in  e.g. \cite[Theorem 4.1]{arnold1997preconditioning}, so we limit ourselves only to sketch a proof here.
		Setting $P = R^1_k \Lambda_k = \sum_{\nu \in \N_k}P^1_{k,\nu}$, the uniform finite overlaps of the domains $\Omega_{k,\nu}$ ensure that
		\eqns{
			\Lambda\inner{P\tau}{P\tau} \leq K_1 \Lambda\inner{Pv}{v},
		}
		for some $K_1$, independent of $k$. It then follows that
		\eqns{
			\inner{R^s_k \Lambda_k \tau}{\Lambda_k \tau} = \Lambda\inner{P\tau}{\tau} \leq \Lambda\inner{P\tau}{P\tau}^{\frac{1}{2}}\Lambda\inner{\tau}{\tau}^{\frac{1}{2}} \leq  \left[ K_1\Lambda\inner{P\tau}{\tau}\right]^{\frac{1}{2}}\Lambda\inner{\tau}{\tau}^{\frac{1}{2}}.
		}
		Replacing $\tau$ with $\Lambda_k^{-1}\tau$ in the above yields 
		the second inequality of \eqref{eq:R_k-upper-bound-01}, 
		and the first inequality of \eqref{eq:R_k-upper-bound-01} can be proved similarly.
		
		For the intermediate result when $s\in (0,1)$ we introduce the auxiliary Hilbert space $\mathbf{V} := \bigoplus_{\nu \in \N_k}V_{k,\nu}$, and  define operators $\mathbf{Q}: V_k \to \mathbf{V}$  and $\mathbf{\Lambda}: \mathbf{V} \to \mathbf{V}$ given by
		\eqns{
			(\mathbf{Q}\tau)_\nu = Q_{k,\nu}\tau, 
		}
		and
		\eqns{
			\left(\mathbf{\Lambda}\boldsymbol{\tau}\right)_\nu = \Lambda_{k,\nu}\boldsymbol{\tau}_\nu,
		}
		for $\tau \in V_k$,  $\nu \in \N_k$, and $\boldsymbol{\tau} \in \mathbf{V}$. In particular, we note that $\mathbf{\Lambda}$ is symmetric positive definite and diagonal on $\mathbf{V}$. 
		Therefore,  $(\mathbf{\Lambda}^\theta \boldsymbol{\tau})_\nu = \Lambda_{k,\nu}^\theta \boldsymbol{\tau}_\nu$ for every $\theta \in \reals$, 
		and so we have $R^s_k = \adjoint{\mathbf{Q}}\mathbf{\Lambda}^{-s}\mathbf{Q}$ for $s\in[0,1]$.
		
		From the definition of $R^0_k$ and the first inequality of \eqref{eq:R_k-upper-bound-01},  $\adjoint{\mathbf{Q}} \mathbf{Q} \leq K_0$, and by scaling $\tilde{\mathbf{Q}} := K_0^{\frac{1}{2}}\mathbf{Q}$ we then have that 
		\eqns{
			\adjoint{\tilde{\mathbf{Q}}}\tilde{\mathbf{Q}} \leq 1.
		}
		Then, by Lemma \ref{lem:gen-jensen-inequality},
		\eqn{
			\label{eq:R_k-upper-bound-proof1}
			R^s_k = K_0 \adjoint{\tilde{\mathbf{Q}}}\mathbf{\Lambda}^{-s}\tilde{\mathbf{Q}} \leq K_0 \left( \adjoint{\tilde{\mathbf{Q}}}\mathbf{\Lambda}^{-1}\tilde{\mathbf{Q}} \right)^s = K_0 \left( K_0^{-1} R^1_k \right)^s = K_0^{1-s}\left(R^1_k\right)^s.
		}
		Now, the second inequality of \eqref{eq:R_k-upper-bound-01} states that $R^1_k \leq K_1 \Lambda_k^{-1}$. Inserting this into \eqref{eq:R_k-upper-bound-proof1} yields
		\eqns{
			R^s_k \leq K_0^{1-s}K_1^s \Lambda_k^{-s},
		}
		which is equivalent to \eqref{eq:R_k-upper-bound-s}.
	\end{proof}
	
	Establishing that Assumption \eqref{eq:Rs-condition2} in Theorem \ref{thm:Bsdiv-spectral-equivalence} holds turns out to be a more complicated matter. In view of \eqref{eq:Rs-inv-identity}, we see that to prove \eqref{eq:Rs-condition2} it is sufficient to find for every $\tau \in (I - P^s_{k,k-1})(V_k)$ a decomposition $\tau = \sum_{\nu} \tau_\nu$, where $\tau_\nu \in V_{k,\nu}$ so that 
	\eqns{
		\sum_{\nu \in \N_k} \inner{\Lambda_{k,\nu}^s \tau_\nu}{\tau_\nu} \leq C \inner{\Lambda_k^s \tau}{\tau},
	}
	for some constant $C$ that is independent of $k$ and $\tau$. 
	In the following lemma, we verify this stable decomposition, assuming some error bounds on the discrete Helmholtz decomposition.
	\begin{lemma}
		\label{lem:Rs_k-lower-bound}
		For $k=2,\ldots, J$, let $\tau \in (I-P^s_{k,k-1})V_k$ have the discrete Helmholtz decomposition
		\eqn{
			\label{eq:error-helmholtz-decomposition}
			\tau = \grad_k u + \curl q,
		}
		for some $u \in S_k$ and $q \in \curl_k V_k$. Assume there exists a constant $c$, independent of $k$ and $\tau$ so that
		\eqn{
			\label{eq:fractional-Hdiv-error-estimates}
			\norm{\grad_k u}^2 \leq c h_{k-1}^{2s}\inner{\Lambda_k^s \tau}{\tau}, \quad \text{ and } \quad \norm{q} \leq c h_{k-1}\norm{\tau}.
		}
		Then there exists a decomposition $\tau = \sum_{\nu \in \N_k} \tau_\nu$ with $\tau_\nu \in V_{k,\nu}$, and a constant $C$ so that
		\eqn{
			\label{eq:Rs_k-lower-bound}
			\sum_{\nu \in \N_k}\inner{\Lambda_{k,\nu}^s \tau_\nu}{\tau_\nu} \leq C \inner{\Lambda_k^s \tau}{\tau}.
		}
	\end{lemma}
	\begin{proof}
		Fix $\tau \in (I-P^s_{k,k-1})V_k$, and let $u \in S_h$ and $q \in \curl_k V_k$ be the discrete Helmholtz decomposition according to \eqref{eq:error-helmholtz-decomposition}. Further, for $\nu \in \N_k$, let $\tilde{\tau}_\nu \in V_{k,\nu}$ and $q_\nu \in C_{k,\nu}$ be $L^2$-stable decompositions of $\grad_k u$ and $q$, respectively. That is, $q = \sum_{\nu}q_\nu$ and $\grad_k u = \sum_{\nu} \tilde{\tau}_\nu$ satisfies
		\eqn{
			\label{eq:Rs_k-lower-bound-proof1}
			\sum_{\nu}\norm{\tilde{\tau}_\nu}^2 \leq c \norm{\grad_k u}^2, \quad \text{ and } \quad \sum_{\nu}\norm{q_\nu}^2 \leq c \norm{q}^2,
		}
		according to \eqref{eq:L2-stable-decomposition}. Then $\tau = \sum_{\nu} \tau_\nu$, where we set $\tau_\nu = \tilde{\tau}_\nu + \curl q_\nu \in V_{k,\nu}$.
		
		By standard inverse inequality, $\inner{\Lambda_{k,\nu} \tilde{\tau}_\nu}{\tilde{\tau}_\nu} \leq c(1 + h_k^{-2})\norm{\tilde{\tau}_\nu}^2$, and the inequality $\inner{\Lambda_{k,\nu}^s \tilde{\tau}_\nu}{\tilde{\tau}_\nu} \leq \inner{\tilde{\tau}_\nu}{\tilde{\tau}_\nu}^{1-s}\inner{\Lambda_{k,\nu}\tilde{\tau}_\nu}{\tilde{\tau}_\nu}^s$ (cf. e.g. \cite[Ch. 2.5]{lions1972nonhom})
		\eqn{
			\label{eq:Rs_k-lower-bound-proof2}
			\inner{\Lambda_{k,\nu}^s\tilde{\tau}_\nu}{\tilde{\tau}_\nu} \leq c (1+h_k^{-2})^s\norm{\tilde{\tau}_\nu}^2.
		}
		Using the fractional inverse inequality \eqref{eq:Rs_k-lower-bound-proof2} and a standard inverse inequality for $q_\nu$, together with \eqref{eq:Rs_k-lower-bound-proof1},
		\eqns{
			\algnd{
				\sum_{\nu}\inner{\Lambda_{k,\nu}^s \tau_\nu}{\tau_\nu} &\leq 2\sum_{\nu}\left[ \inner{\Lambda_{k,\nu}^s\tilde{\tau}_\nu}{\tilde{\tau}_\nu} + \norm{\curl q_\nu}^2 \right] \\
				&\leq c \sum_{\nu}\left[(1+h_k^{-2})^s \norm{\tilde{\tau}_\nu}^2 + h_k^{-2}\norm{q_\nu}^2 \right] \\
				&\leq c\left[(1+h_k^{-2})^s\norm{\grad_k u}^2 + h_k^{-2}\norm{q}^2 \right]. 
			}
		}
		Then, \eqref{eq:Rs_k-lower-bound} follows from the above and assumption \eqref{eq:fractional-Hdiv-error-estimates}.
	\end{proof}
	\begin{rem}
		Verifying the assumption of Lemma \ref{lem:Rs_k-lower-bound} is by no means a trivial matter, 
		and falls beyond the scope of this paper. 
		As such, we leave the additive multigrid operators $B^s_{\div, h}$ 
		on what may be deemed an unsure theoretical footing.
		However, we will here propose an approach to prove the assumptions made in 
		Lemma \ref{lem:Rs_k-lower-bound}. 
		First off, the case $s=1$ was proved in \cite{arnold1997preconditioning,arnold2000multigrid}, 
		where the thrust of the argument relied on two-level error estimates and duality arguments.
		
		From the identity \eqref{eq:AP=QA-Hdiv} we see that
		\eqns{
			P^s_k - P^s_{k-1} = (\Lambda_k^{-s} - \Lambda_{k-1}^{-s}Q_{k-1})Q_k\Lambda_h^s,
		}
		and so the assumptions in Lemma \ref{lem:Rs_k-lower-bound} are concerned with 
		two-level error estimates for discretizations of fractional $H(\div)$ problems. 
		For the first estimate of \eqref{eq:error-helmholtz-decomposition}, 
		we recall the observation that $\Lambda_k^s$ behaves like an elliptic operator on $\grad_k S_k$, 
		and so the required error estimate can be obtained using 
		similar techniques as in \cite[Thm. 4.3]{bonito2015numerical}. 
		There, the authors proved error estimates, under some regularity assumptions on the domain $\Omega$. 
		The proof uses the integral formulation of the fractional Laplacian,
		\eqns{
			(-\Delta)^s = \frac{2\sin (\pi s)}{\pi}\int_0^\infty t^{2s-1}(I - t^2\Delta)^{-1} \intd t. 
		}
		See also \cite[Sec. 10.4]{birman1987spectral}.
		The advantage of this approach is that error estimates for the fractional Laplacian are transferred 
		to error estimates for problems of the form
		\eqns{
			(I - t^2 \Delta) u = f,
		}
		where an abundance of results are available.
	\end{rem}
	
	\section{Numerical experiments}
	\label{sec:numerical_experiments}
	We now present a series of numerical experiments, aimed at validating the theoretical results established in previous sections. Specifically, in section \ref{sec:fractional-hdiv} we test the preconditioner $B^s_{\div,h}$ defined in \eqref{eq:Bsdiv-def}, and the spectral equivalence established in Theorem \ref{thm:Bsdiv-spectral-equivalence}. 
	
	We then consider 
	\eqn{
		\label{eq:numex-fraclap-prob}
		A_h^{s} u = f,
	}
	for a given in $s \in [-1,0]$ and $f \in S_h$.
	In section \ref{sec:exact_auxiliary_space_preconditioner}, \eqref{eq:numex-fraclap-prob} is first solved using $\adjoint{\grad_h} \Lambda_h^{-(1+s)}\grad_h$ as preconditioner, before we use $B^s_h$ defined in \eqref{eq:auxSpace-discrete-PC-def} as preconditioner. These experiments are to validate Theorem \ref{thm:auxSpace-discrete-PC-speceq} and Corollary \ref{cor:discrete-auxSpace-speceq}, respectively.
	
	Where applicable, the numerical tests are conducted using preconditioned conjugate method, with random initial guess. Convergence of the iterations are reached when the relative preconditioned residual, i.e. $\frac{\inner{B r_k}{r_k}}{\inner{B r_0}{r_0}}$, where $r_k$ is the $k$-th residual and $B$ is the preconditioner, is below a given tolerance.
	
	Note that in the following, all fractional powers of matrices are constructed by full spectral decomposition, requiring the solution of large generalized eigenvalue problems (see \cite{kuchta2016trace2d} for details). 
	As such, the preconditioned iterative methods will not be computationally optimal, but the tests are designed only to validate the theoretical bounds on the condition numbers. 
	This problem will not be encountered if $B^s_h$ is used as part of a preconditioner for trace problems as presented in the introduction. 
	
	\subsection{Preconditioning for $\Lambda_h^s$}
	\label{sec:fractional-hdiv}
	In the first set of numerical experiments we consider the following problem: For a given $f \in V_h$ and $s\in [0,1]$, find $\sigma \in V_h$ so that
	\eqn{
		\label{eq:numex-fractional-hdiv}
		\Lambda_h^s \sigma = f.
	}
	We take $\Omega = [0,1]^2 \subset \reals^2$, and $\T_h$ is a uniform partition of $\Omega$. We take $V_h$ to be the lowest order Raviart-Thomas space relative to the mesh $\T_h$. We solve the linear system arising from \eqref{eq:numex-fractional-hdiv} using preconditioned conjugate gradient method, with $B^s_{\div,h}$ given by \eqref{eq:Bsdiv-def} as preconditioner. The results can be seen in Table \ref{tab:hdiv-additive-MG}, from which we see that both iteration counts and condition numbers stay bounded independently of the dimension of $V_h$, in accordance with Theorem \ref{thm:Bsdiv-spectral-equivalence}.
	
	\begin{table}[t]
		\begin{tabular}{l | l l l l }
			\diagbox{$s$}{$N$} & 208& 800& 3136& 12416\\
			\hline
			$0.0$& $20(4.9)$& $21(4.9)$& $21(4.9)$& $21(4.9)$\\ 
			$0.1$& $20(4.6)$& $21(4.9)$& $22(5.2)$& $23(5.5)$\\ 
			$0.2$& $22(5.6)$& $24(6.2)$& $25(6.8)$& $27(7.4)$\\ 
			$0.3$& $24(6.6)$& $26(7.5)$& $27(8.1)$& $28(8.6)$\\ 
			$0.4$& $26(8.0)$& $28(8.7)$& $29(9.2)$& $29(9.6)$\\ 
			$0.5$& $27(9.2)$& $30(9.8)$& $30(10.3)$& $30(10.5)$\\ 
			$0.6$& $29(10.4)$& $31(10.9)$& $31(11.3)$& $31(11.5)$\\ 
			$0.7$& $30(11.6)$& $32(12.1)$& $32(12.4)$& $32(12.5)$\\ 
			$0.8$& $31(13.0)$& $33(13.4)$& $33(13.5)$& $33(13.7)$\\ 
			$0.9$& $32(14.5)$& $35(14.9)$& $34(14.9)$& $34(15.0)$\\ 
			$1.0$& $33(16.1)$& $36(16.5)$& $36(16.6)$& $35(16.5)$\\ 
		\end{tabular}
		\caption{Numerical results preconditioning $\Lambda_h^s$. Table show number of conjugate gradient interations until reaching relative error tolerance $10^{-9}$. Estimated condition numbers are shown inside parentheses. $N = \dim V_h$ and  $J = 4$ in all tests.}
		\label{tab:hdiv-additive-MG}
	\end{table}
	
	\subsection{Auxiliary space preconditioner}
	\label{sec:exact_auxiliary_space_preconditioner}
	We now consider \eqref{eq:numex-fraclap-prob} on the same computational domain as in the previous set of experiments. That is, $\Omega = [0,1]^2$, and $\T_h$ is a uniform triangulation of $\Omega$. For the discrete space $S_h$ we use piecewise constants relative to $\T_h$. In Table \ref{tab:exact_auxiliary_space_preconditioner}, we can view the calculated condition number of $\adjoint{\grad_h} \Lambda_h^{-(1+s)}\grad_h\, A^s_h$, as well as the condition number expected from Theorem \ref{thm:auxSpace-discrete-PC-speceq}. The results show both uniform $h$-independence and is in good agreement with the theory.
	
	Finally, we solve \eqref{eq:numex-fraclap-prob} using preconditioned conjugate gradient method, with $B^s_h = \adjoint{\grad_h} B^{1+s}_{\div,h} \grad_h$ defined in \eqref{eq:auxSpace-discrete-PC-def} as preconditioner. $B^{1+s}_{\div,h}$ is chosen as the additive multigrid operator proposed in section \ref{sec:decompositions}, with $V_h$ as the lowest order Raviart-Thomas space relative to $\T_h$. The results can be viewed in Table \ref{tab:neg_fraclap_preconditioner}. Again, we see that both iteration counts and estimated condition numbers stay reasonably bounded, in agreement with Corollary \ref{cor:discrete-auxSpace-speceq}, although a slight increase becomes pronounced as $s$ approaches $0$.
	\begin{table}[t]
		\begin{tabular}{l | l l l | l}
			\diagbox{$s$}{$N$} & 512& 2048& 8192 & $\beta^{-2(1+s)}$\\
			\hline
			$-1.0$&  $1.000$& $1.000$& $1.000$ & $1.000$\\ 
			$-0.9$& $1.005$& $1.005$& $1.005$ & $1.005$ \\ 
			$-0.8$& $1.010$& $1.010$& $1.010$ & $1.010$\\ 
			$-0.7$& $1.015$& $1.015$& $1.015$ & $1.015$\\ 
			$-0.6$& $1.020$& $1.020$& $1.020$ & $1.020$\\ 
			$-0.5$& $1.025$& $1.025$& $1.025$ & $1.025$\\ 
			$-0.4$& $1.030$& $1.030$& $1.030$ & $1.030$ \\ 
			$-0.3$& $1.035$& $1.035$& $1.035$ & $1.035$\\ 
			$-0.2$& $1.040$& $1.040$& $1.040$& $1.041$\\ 
			$-0.1$& $1.045$& $1.045$& $1.045$& $1.046$\\ 
			$0.0$&  $1.050$& $1.051$& $1.051$& $1.051$\\ 
		\end{tabular}
		\caption{Numerical results for exact auxiliary space preconditioner. Table show condition number of $\grad_h^* \Lambda_h^{{-(1+s)}} \grad_h A^s_h$. $N$ is dimension of $S_h$. The rightmost column shows expected condition number from Theorem \ref{thm:auxSpace-discrete-PC-speceq} with $\beta^{-2} = 1.051$.  }
		\label{tab:exact_auxiliary_space_preconditioner}
	\end{table}
	
	\begin{table}[t]
		\begin{tabular}{l | l l l l }
			\diagbox{$s$}{$N$} & 128& 512& 2048& 8192\\
			\hline
			$-1.0$& $18(4.3)$& $19(4.4)$& $20(4.6)$& $21(4.6)$\\ 
			$-0.9$& $17(3.7)$& $19(3.7)$& $19(3.7)$& $19(3.7)$\\ 
			$-0.8$& $17(3.2)$& $18(3.2)$& $18(3.2)$& $18(3.2)$\\ 
			$-0.7$& $17(2.9)$& $18(2.9)$& $18(2.9)$& $18(3.0)$\\ 
			$-0.6$& $17(2.8)$& $18(3.0)$& $18(3.1)$& $19(3.1)$\\ 
			$-0.5$& $18(3.2)$& $19(3.3)$& $20(3.4)$& $20(3.6)$\\ 
			$-0.4$& $19(3.6)$& $21(3.8)$& $21(3.8)$& $22(4.4)$\\ 
			$-0.3$& $19(4.0)$& $22(4.2)$& $22(4.2)$& $24(5.3)$\\ 
			$-0.2$& $20(4.5)$& $23(4.8)$& $24(5.1)$& $26(6.2)$\\ 
			$-0.1$& $21(5.1)$& $25(5.4)$& $26(6.1)$& $28(7.2)$\\ 
			$0.0$& $22(5.8)$& $27(6.2)$& $28(7.4)$& $30(8.3)$\\ 
		\end{tabular}
		\caption{Numerical results for preconditioning $A_h^s$ with $B^s_h$ given by \eqref{eq:auxSpace-discrete-PC-def}. Table show number of conjugate gradient interations until reaching error tolerance $10^{-10}$. Estimated condition numbers are shown inside parentheses. $N$ is dimension of $S_h$. $J = 4$ in all tests.}
		\label{tab:neg_fraclap_preconditioner}
	\end{table}
\newpage
\clearpage
\bibliographystyle{abbrv}
\vspace{.125in}
\bibliography{arxiv_fractional_pc}

\end{document}